\documentclass[12pt,twoside]{amsart}

\usepackage[latin1]{inputenc}
\usepackage[OT1]{fontenc}
\usepackage{amsmath}
\usepackage{amsthm}
\usepackage{bbm}
\usepackage{amssymb}
\usepackage{xcolor}
\usepackage[all]{xy}
\usepackage{hyperref}

\newtheorem{thm}{Theorem}[section]

\newtheorem{proposition}[thm]{Proposition}

\newtheorem{lemma}[thm]{Lemma}
\newtheorem{corollary}[thm]{Corollary}

\numberwithin{equation}{section}

\theoremstyle{definition}

\newtheorem{remark}[thm]{Remark}

\newcommand{\qqed}{\hspace*{\fill}$\Box$}

\newcommand{\Db}{{\rm D}^{\rm b}}

\newcommand{\Aut}{{\rm Aut}}

\newcommand{\Hom}{{\rm Hom}}
\DeclareMathOperator{\Spec}{{\rm Spec}}

\newcommand{\id}{{\rm id}}

\newcommand{\Ext}{{\rm Ext}}

\newcommand{\ka}{{\mathcal A}}
\newcommand{\kb}{{\mathcal B}}
\newcommand{\kc}{{\mathcal C}}

\newcommand{\ki}{{\mathcal I}}

\newcommand{\kl}{{\mathcal L}}

\newcommand{\ko}{{\mathcal O}}

\newcommand{\CC}{\mathbb{C}}

\newcommand{\PP}{\mathbb{P}}
\newcommand{\QQ}{\mathbb{Q}}

\newcommand{\ZZ}{\mathbb{Z}}

\renewcommand{\to}{\xymatrix@1@=15pt{\ar[r]&}}
\renewcommand{\rightarrow}{\xymatrix@1@=15pt{\ar[r]&}}
\renewcommand{\mapsto}{\xymatrix@1@=15pt{\ar@{|->}[r]&}}
\renewcommand{\twoheadrightarrow}{\xymatrix@1@=15pt{\ar@{->>}[r]&}}
\renewcommand{\hookrightarrow}{\xymatrix@1@=15pt{\ar@{^(->}[r]&}}
\newcommand{\congpf}{\xymatrix@1@=15pt{\ar[r]^-\sim&}}
\renewcommand{\cong}{\simeq}

\begin{document}

\title{Some remarks on phantom categories and motives}
\author{Pawel Sosna}
\address{Fachbereich Mathematik der Universit\"at Hamburg\\
Bundesstra\ss e 55\\
20146 Hamburg, Germany}
\email{pawel.sosna@math.uni-hamburg.de}

\begin{abstract}
A phantom category is an admissible subcategory with vanishing Grothendieck group of the bounded derived category of coherent sheaves on a smooth projective variety. The goal of this paper is to study the abstract situation when such a category appears and establish some results which provide evidence for the idea that these categories are invisible on the level of Chow motives.
\end{abstract}
\maketitle

\section{Introduction}

The bounded derived category $\Db(X)$ of a smooth projective variety $X$ has for quite some time been recognized as an interesting invariant encoding a lot of geometric information. For example, there are results, and conjectures, linking semi-orthogonal decompositions of $\Db(X)$ to the birational geometry or to the structure of the Chow and/or the noncommutative motive of $X$, see, for example, \cite{Bern-Tab}, \cite{Kuz10} or \cite{Marc-Tab}. 

Recently, very special examples of semi-orthogonal decompositions were constructed. Namely, it was shown for several complex surfaces $X$ that $\Db(X)$ admits a semi-orthogonal decomposition consisting of an exceptional collection of line bundles whose orthogonal complement is a category with trivial Hochschild homology and finite or trivial Grothendieck group; see, for example, \cite{A-O12}, \cite{BBS12}, \cite{BBKS}, \cite{Cho-Lee}, \cite{GS13}, \cite{Gor-Orl} or \cite{Lee15}. In the first case we call such a category a quasi-phantom and in the second case a phantom. There is also the notion of a universal phantom category, see Subsection \ref{sDerCat}.

Instead of adding to the list of examples of (quasi-)phantom categories, in this paper the abstract situation when such a category appears is studied. In all known examples where this happens, the Grothendieck group of the variety is of finite rank, and in Proposition \ref{pGrothExc} we show that this already implies that the structure sheaf is an exceptional object. In Propositions \ref{pPhantomChow} and \ref{pPhantom-Lefschetz} we explore, under certain assumptions on the semi-orthogonal decomposition, some implications for the structure of the Chow motive of a variety admitting a (universal) phantom category. We also investigate when a phantom category is automatically a universal phantom category, see Proposition \ref{pPhan-uPhan}.

The paper is organized as follows. The relevant notions and results concerning semi-orthogonal decompositions, Chow and $K$-motives are collected in Section 2. In Section 3 we prove Proposition \ref{pGrothExc}. In the following section we establish the results alluded to above. In the final section we make some comments on the Hochschild cohomology when products of (quasi-)phantom categories are taken and point out some possible future directions of research.
\smallskip

\noindent\textbf{Conventions.} Unless stated otherwise we work in the category of smooth projective varieties over some field $k$. All functors are assumed to be derived.
\smallskip

\noindent\textbf{Acknowledgements.} I thank Daniel Huybrechts and Andreas Krug for useful comments and suggestions on a preliminary version of the paper and Sergey Galkin and Charles Vial for their comments after the paper was put on the arXiv. I was partially financially supported by the RTG 1670 of the DFG (German Research Foundation).

\section{Preliminaries}
\subsection{Derived categories}\label{sDerCat}
An object $E\in \Db(X)$ is called \emph{exceptional} if $\Hom^0(E,E)=k$ and $\Hom^i(E,E)=0$ for all $i\neq 0$. An ordered sequence of exceptional objects $(E_1,\ldots,E_r)$ is called an \emph{exceptional collection} if $\Hom^l(E_j,E_i)=0$ for all $j>i$ and for all $l$. An exceptional collection is called \emph{full} if the smallest triangulated subcategory of $\Db(X)$ containing all the $E_i$ is equivalent to $\Db(X)$. 	

The above notions fit into the more general framework of semi-orthogonal decompositions which were introduced in \cite{Bon-Kap}.
Namely, a \emph{semi-orthogonal decomposition} (s.d.) of $\Db(X)$ is a sequence of strictly full triangulated subcategories $\ka_1,\ldots,\ka_m$ such that (a) if $A_i\in \mathcal{A}_i$ and $A_j\in \mathcal{A}_j$, then $\text{Hom}(A_i,A_j[l])=0$ for $i>j$ and all $l$, and (b) the $\mathcal{A}_i$ generate $\Db(X)$, that is, the smallest triangulated subcategory of $\Db(X)t$ containing all the $\mathcal{A}_i$ is already $\Db(X)$. We write $\Db(X)=\langle\ka_1,\ldots,\ka_m\rangle$. If $m=2$, these conditions boil down to the existence of a functorial exact triangle $A_2\to D\to A_1\to A_2[1]$ for any object $D \in \Db(X)$.

A full subcategory of $\Db(X)$ is called \emph{admissible} if the embedding functor has a left and a right adjoint. Since $\mathsf{S}=-\otimes\omega_X[\dim(X)]$ is a Serre functor on $\Db(X)$, that is, there are bifunctorial isomorphisms $\Hom(D,D')\cong \Hom(D',\mathsf{S}(D))^\vee$, any admissible subcategory of $\Db(X)$ also has a Serre functor. 

It can be checked that in the situation of a semi-orthogonal decomposition of $\Db(X)$ the subcategories involved are admissible. Conversely, if $\kc$ is admissible in $\Db(X)$, then $\Db(X)=\langle \kc^\bot,\kc\rangle$, where $\kc^\bot=\{D\in \Db(X)\mid \Hom(C,D)=0\; \forall\, C\in \kc\}$ is the (right) \emph{orthogonal complement} to $\kc$. 

One can see rather easily that if $E\in \Db(X)$ is an exceptional object, then the category generated by $E$, which is equivalent to $\Db(\mathrm{Spec}k)$, is admissible, see \cite[Lem.\ 1.58]{Huy-book}. Abusing notation one denotes this category again by $E$.

Given two semi-orthogonal decompositions $\Db(X)=\langle \ka_1,\ldots,\ka_m\rangle$ and $\Db(Y)=\langle \kb_1,\ldots,\kb_n\rangle$, there is an induced semi-orthogonal decomposition of the product given by
\[\Db(X\times Y)=\langle (\ka_i\boxtimes\kb_j)_{i,j}\rangle,\]
where $\ka_i\boxtimes\kb_j$ denotes the smallest triangulated subcategory of $\Db(X\times Y)$ containing all objects of the form $p^*A\otimes q^*B$ (where $p\colon X\times Y\to X$ and $q\colon X\times Y\to Y$ are the respective projections) for $A\in \ka_i$ and $B\in \kb_j$, and closed under direct summands; see \cite[Prop.\ 1.6]{Gor-Orl} or \cite[Thm.\ 5.8]{Kuzbase}. 

If $\Db(X)=\langle \ka_1,\ldots,\ka_m\rangle$, then the adjoints to the embeddings give an isomorphism $K_0(X)=K_0(\Db(X))\cong \oplus_{i=1}^mK_0(\ka_i)$, where $K_0(-)$ denotes the Grothendieck group.

An admissible subcategory $\ka$ of $\Db(X)$ is called a \emph{phantom category} if $K_0(\ka)=0$ and a \emph{universal phantom category} if $K_0(\ka\boxtimes \Db(Y))=0$ for any $Y$. If $\Db(X)=\langle \ka, E_1,\ldots,E_r\rangle$ is an s.d.\ where the $E_i$ are all exceptional and $\ka$ is a (universal) phantom category, we call the collection $(E_1,\ldots,E_r)$ \emph{almost full}.

\subsection{Motives}
One possible reference for the following two subsections is \cite{Manin}.

The category of Chow motives $\mathrm{CM}(k)$ has as objects triples $(X,\pi,n)$, where $X$ is smooth projective, $\pi\in \mathrm{CH}^{\dim X}(X\times X)$ is a cycle with $\pi\circ \pi=\pi$ and $n\in \ZZ$ (as usual, we write $\mathrm{CH}^n(X)$ for codimension $n$ cycles). Here, $\pi\circ\pi$ is the usual convolution product. The morphisms in $\mathrm{CM}(k)$ are given by 
\[\Hom_{\mathrm{CM}(k)}((X,\pi,n),(Y,\psi,m))=\psi\circ \mathrm{CH}^{\dim X+m-n}(X\times Y)\circ\pi.\]
There is a contravariant functor from varieties to $\mathrm{CM}(k)$ which sends $X$ to $M(X)=(X,\Delta_X,0)$. The \emph{unit motive} is $M(\Spec k)=\mathbbm{1}=(\Spec k,\Delta,0)$ and the \emph{Lefschetz motive} is $\mathbb{L}=(\Spec k,\Delta,-1)$. 

It is easy to check that $M(\PP^1)=\mathbbm{1}\oplus \mathbb{L}$. Note that $\mathrm{CM}(k)$ is a symmetric monoidal category with the tensor product induced by the fibre product of varieties. The unit with respect to this tensor structure is precisely $\mathbbm{1}$.

In fact, all of this can be done over any extension ring $R$ of $\ZZ$, that is, one replaces integral Chow groups by Chow groups with $R$-coefficients.

Probably the simplest motives are those of \emph{Lefschetz type}, that is, those which are isomorphic to a direct sum of finitely many motives $\mathbb{L}^{\otimes p}$ for some integers $p$.

\subsection{$K$-motives}
The category of $K$-motives $\mathrm{KM}(k)$ is defined similarly as the category of Chow motives but with Grothendieck groups replacing cycles. For instance, objects of $\mathrm{KM}(k)$ are pairs $(X,\pi)$ where $\pi\in K_0(X\times X)$ is a projector. Yet again, we have a contravariant functor $\mathrm{KM}$ from varieties to $K$-motives. The category $\mathrm{KM}(k)$ is symmetric monoidal and the unit object is $\mathbf{1}=\mathrm{KM}(\Spec k)=(\Spec k,[\Delta])$. A $K$-motive is called of \emph{trivial type} if it is isomorphic to a finite direct sum $\oplus_{i=1}^m \mathbf{1}$. Of course, in this setting we can also consider extension rings of $\ZZ$.

In fact, $K$-motives fit into the more general framework of \emph{noncommutative motives} defined using the language of differential graded categories; see \cite{Marc-Tab}. In particular, the category $\mathrm{KM}(k)$ embeds into the category of noncommutative motives. Furthermore, we can associate a $K$-motive to any admissible subcategory of $\Db(X)$. The following is an easy consequence of the construction.

\begin{lemma}
Assume there are semi-orthogonal decompositions $\Db(X)=\langle \ka_1,\ldots,\ka_m\rangle$ and $\Db(Y)=\langle \kb_1,\ldots,\kb_n\rangle$ and consider the induced s.d.\
\[\Db(X\times Y)=\langle (\ka_i\boxtimes\kb_j)_{i,j}\rangle.\]
Then $\Hom_{\mathrm{KM}(k)}(\mathrm{KM}(\ka_i),\mathrm{KM}(\kb_j))=K_0(\kb_j\boxtimes \ka_i)$.
\end{lemma} 

\begin{proof}
We only briefly outline the proof, refering to \cite[Sect.\ 4]{Gor-Orl} for details.

Denote the inclusion $\ka_i\to \Db(X)$ by $j_{\ka_i}$. Since $\ka_i$ is admissible, it has a right adjoint $j_{\ka_i}^R$. The functor $\Phi=j_{\ka_i}\circ j_{\ka_i}^R$ is a Fourier-Mukai functor by \cite[Thm.\ 7.1]{Kuzbase}, that is, there exists an object $K\in \Db(X\times X)$ such that $\Phi=\mathrm{FM}_K$, where $\mathrm{FM}_K=p_{2*}(K\otimes p_1^*(-))$.

By definition, $\mathrm{KM}(\ka_i)=(X,[K])$. The claim follows from this.
\end{proof}

Another consequence of the construction is the isomorphism 
\[\mathrm{KM}(\ka_i\boxtimes \kb_j)\cong \mathrm{KM}(\ka_i)\otimes\mathrm{KM}(\kb_j).\]

Concerning the interplay between Chow motives and $K$-motives, we know by \cite[Prop.\ 4.2]{Gor-Orl} that if the Chow motive of a smooth projective variety $X$ of dimension $n$ over a field of characteristic zero is of Lefschetz type over a ring $R\supseteq\ZZ$, then the $K$-motive is of trivial type over $R$. The converse was known to hold over $\QQ$ by \cite{Marc-Tab} and was recently shown to already hold over $R[1/(2n)!]$, albeit under some assumptions on $R$; see \cite[Thm.\ 1.4]{Bern-Tab}. 

\section{Phantom-free results}

So far, phantoms have been constructed either as orthogonal complements to almost full exceptional collections on some complex surfaces or on products of such surfaces. In particular, in all examples the Grothendieck group is of finite rank. This has the following curious consequence.

\begin{proposition}\label{pGrothExc}
Let $X$ be a smooth complex projective variety of dimension $n$ whose Grothendieck group is of finite rank. Then $\ko_X$ is an exceptional object.
\end{proposition}

\begin{proof}
Recall that $\mathrm{CH}^\bullet(X)\otimes\QQ\cong K_0(X)\otimes\QQ$, hence the Chow ring is of finite rank. In particular, $\mathrm{CH}^n(X)=\mathrm{CH}_0(X)$ is a finitely generated abelian group. Therefore, the kernel $\mathrm{CH}^n(X)_{\text{hom}}=\mathrm{CH}_0(X)_{\text{hom}}$ of the cycle map $cl\colon \mathrm{CH}^n(X)\to H^{2n}(X,\ZZ)\cong \ZZ$ is a finitely generated abelian group. In particular, by a theorem of Roitman, the Albanese map gives an isomorphism $\mathrm{CH}_0(X)_{\text{hom}}\cong \mathrm{Alb}(X)$, see \cite[Thm.\ 10.2]{Voisin-book2}.

On the other hand, the dimension of the Albanese variety of $X$ has to be zero by our assumption, hence $\mathrm{Alb}(X)=0$, so $\mathrm{CH}_0(X)_{\text{hom}}$ is trivial. Then, by \cite[Cor.\ 10.18]{Voisin-book2}, $H^0(X,\Omega^k_X)=0$ for all $k>0$. By Hodge symmetry, we conclude that $h^{0,k}=0$ for all $k>0$, that is, $\ko_X$ is exceptional.  
\end{proof}

\begin{remark}
The case of curves is trivial in this context. In the case of a complex surface $S$, the converse holds if we assume the Bloch conjecture. We can reformulate this more categorically by saying that $F^2K_0(S)\cong \ZZ$, where $F^2K_0(S)$ is the subgroup generated by sheaves supported in points. 

It is not clear to the author what assumption one has to add in general to prove the converse. For example, the structure sheaf of a cubic fourfold $X$ is exceptional, but the Grothendieck group is far from being of finite rank. Indeed, $\Db(X)=\langle \ka,\ko,\ko(1),\ko(2)\rangle$, where $\ka$ is a K3 category, that is, its Serre functor is shift by 2 and its Hochschild (co)homology coincides with that of a K3 surface. If $X$ is known to be rational, then $\ka$ is equivalent to $\Db(S)$ for a K3 surface $S$. By a theorem of Mumford, see \cite[Thm.\ 10.1]{Voisin-book2}, the Grothendieck group of $S$ is far from being of finite rank, so the same holds for $X$.
\end{remark}

\begin{remark}
In some sense, exceptional objects are the simplest ones in $\Db(X)$. Other reasonably simple objects, at least in terms of their endomorphism algebras, are the following.
\begin{enumerate}
\item If $\Hom^\bullet(E,E)=\CC\oplus \CC[-d]$, we call $E$ a \emph{$d$-spherelike object}.  
\item If $E$ is $d$-spherelike and $\mathsf{S}(E)\cong E[d]$, we call $E$ a \emph{$d$-spherical object}. These are interesting since any such object induces an autoequivalence of $\Db(X)$; see \cite{Seid-Thom}.
\item If $\mathsf{S}(E)\cong E[2n]$ and $\Hom^\bullet(E,E)\cong H^*(\PP^n,\CC)$, we call $E$ a \emph{$\PP^n$-object}. These objects also induce autoequivalences; see \cite{Huyb-Thomas}. 
\end{enumerate}

Now, if $X$ is $d$-dimensional, then $\ko_X$ is $d$-spherical if and only if $X$ is strict Calabi-Yau, that is, $\omega_X\cong \ko_X$ and $H^i(\ko_X)=\CC$ for $i=0,\dim(X)$ and $0$ otherwise. Andreas Krug observed that $\ko_X$ is a $\PP^n$-object if and only if $X$ is an irreducible holomorphic symplectic variety of dimension $2n$. It seems that there is no good characterisation when $\ko_X$ is $d$-spherelike. In fact, the structure sheaf of a K3 surface is $2$-spherical, hence in particular also $2$-spherelike. But there also exist surfaces of general type whose structure sheaf is $2$-spherelike. Blowing up a Calabi-Yau variety in points gives examples where $\ko_X$ is $d$-spherelike and not $d$-spherical. See \cite{HKP} for the proofs of the last statements. 

It is somewhat peculiar that it is probably quite difficult to categorically characterise the situation where $\ko_X$ is exceptional.
\end{remark}

\begin{remark}
Note that if $K_0(X)$ is torsion free, then $\mathrm{CH}^1(X)$ is torsion free by \cite[Lem.\ 2.2]{GKMS}. If $K_0(X)$ is torsion free and of finite rank and $X$ is defined over an algebraically closed field, then $\ko_X$ is exceptional, hence by a theorem of Roitman, see \cite[Thm.\ 10.14]{Voisin-book2}, $\mathrm{CH}^{\dim X}(X)$ is torsion free. 

In particular, if $X$ has a full exceptional collection (and is defined over $k=\overline{k}$), then $\mathrm{CH}^0(X)$, $\mathrm{CH}^1(X)$ and $\mathrm{CH}^{\dim(X)}(X)$ are torsion free. A similar statement holds if the collection is only almost full.
\end{remark}

\section{On phantoms and motives}

We begin our investigation of phantom categories with the following easy result which just says that universal phantoms behave well with respect to base change and embeddings.

\begin{proposition}\label{pEasyresults}
Let $X$ be a smooth projective variety over a field $k$ with a semi-orthogonal decomposition $\Db(X)=\langle \ka,\kb\rangle$, where $\ka$ is a universal phantom category. 
\begin{enumerate}
\item Assume there exists a fully faithful functor $\Db(X)\to \Db(Y)$. Then $\ka\subset \Db(Y)$ is a universal phantom category. In particular, this holds for blow-ups, that is, if $Y=\mathrm{Bl}_ZX$ or $Y=\mathrm{Bl}_XW$.
\item Given any smooth projective $Y$, $\ka\boxtimes \Db(Y)\subset \Db(X\times Y)$ is a universal phantom category.
\end{enumerate}
\end{proposition}

\begin{proof}
First note that by Orlov's formula, see \cite{Orl-blowup}, which describes the derived category of a blow-up as a semi-orthogonal decomposition of the derived category of the base and several copies of the derived category of the center of the blow-up (the number of copies depends on the codimension of the center), the statement about blow-ups indeed follows from item (1).
  
Both items follow immediately from \cite[Prop.\ 4.4]{Gor-Orl}, which in particular establishes that $\ka$ is a universal phantom category if and only if $\mathrm{KM}(\ka)=0$, and the fact that $\mathrm{KM}(k)$ is a symmetric monoidal category. Note that the same argument shows that if $\kc\subset \Db(Y)$ is admissible, then $\ka\boxtimes \kc \subset \Db(X\times Y)$ is a universal phantom.

A down-to-earth proof of (2) can be given as follows. First note that $K_0(\ka\boxtimes\Db(Y))=0$ by the universality of $\ka$, so $\ka\boxtimes \Db(Y)$ is a phantom. To see that it is universal, we only need to check that 
\[K_0((\ka\boxtimes \Db(Y))\boxtimes\Db(Z))=0\quad \text{ for any }Z.\]
Recall that $(\ka\boxtimes \Db(Y))\boxtimes\Db(Z)$ is the smallest triangulated subcategory of $\Db(X\times Y\times Z)$ containing all objects of the form $\pi_{XY}^*(A)\otimes \pi_Z^*(F)$, where $A\in \ka\boxtimes\Db(Y)$ and $F\in \Db(Z)$, and closed under direct summands. Now, $\Db(Y\times Z)$ is split generated (that is, we get everything after taking direct summands) by elements of the form $p^*F\otimes q^*G$ (where $F\in \Db(X)$, $G\in \Db(Y)$), and $A$ arises as a direct summand of an object of the form $p^*(A')\otimes q^*G$ $(A'\in \ka$, $G\in \Db(Y)$), hence 
\[(\ka\boxtimes \Db(Y))\boxtimes\Db(Z)\cong \ka\boxtimes\Db(Y\times Z).\]
Therefore, the Grothendieck group of this category vanishes by the universality of $\ka$.
\end{proof}

One particular example of varieties satisfying the assumption of Proposition \ref{pGrothExc} are varieties with almost full exceptional sequences. One easy property these have is given by

\begin{proposition}\label{pHodgenumbers}
Let $X$ be a smooth complex projective variety such that $\Db(X) = \langle \ka,E_1,\ldots,E_r\rangle$, where $\ka$ is a phantom category and $E_i$ are exceptional. Then $h^{p,q}= 0$ if $p\neq q$.
\end{proposition}

\begin{proof}
This follows from the Hochschild-Kostant-Rosenberg
isomorphism
\[\mathrm{HH}_k(X) =\bigoplus\limits_{q-p=k}H^{p,q}(X),\]
the fact that $\mathrm{HH}_\bullet(E_i)=\mathrm{HH}_0(E_i)=\CC$ and additivity of Hochschild homology on semi-orthogonal decompositions; see \cite{Kuz09} for the latter statement, where it is also explained how to define Hochschild homology of an admissible subcategory of $\Db(X)$.
\end{proof}

The next result strengthens \cite[Thm.\ 1.3]{Marc-Tab}, at least in the context of complex varieties. Recall that an admissible subcategory $\ka$ of $\Db(X)$ is a \emph{quasi-phantom category} if $K_0(\ka)$ is finite and $\mathsf{HH}_\bullet(\ka)=0$. Note that Proposition \ref{pHodgenumbers} holds for quasi-phantom categories as well.

\begin{proposition}\label{pPhantomChow}
Let $X$ be a smooth complex projective variety such that $\Db(X) = \langle \ka,E_1,\ldots,E_r\rangle$, where $\ka$ is a (quasi-)phantom category. Then the rational Chow motive of $X$ is of Lefschetz type.
\end{proposition}

\begin{proof}
By \cite[Thm.\ 5]{Vial-proj}, all the rational Chow groups of a smooth projective complex $n$-dimensional variety are finite-dimensional $\QQ$-vector spaces if and only if the Chow motive of $X$ is of the form $M(X)=\oplus_{i=0}^n(\mathbb{L}^i)^{\oplus b_{2i}}$, where $b_{2i}$ is the $2i$-th Betti number. Since $K_0(X)\otimes \QQ\cong \QQ^r\cong \mathrm{CH}_\bullet(X)\otimes \QQ$, the claim follows.
\end{proof}

\begin{remark}\label{rRatGrothGr}
Note that under the assumptions of the proposition the rational Chow motive of $X\times X$ is again of Lefschetz type and \cite[Prop.\ 4.1(ii)]{Gor-Orl} implies that $K_0(X\times X)_\QQ\cong K_0(X)_\QQ\otimes K_0(X)_\QQ$. 
\end{remark}

Next we want to investigate when a phantom category is also a universal phantom. For this we will need the following

\begin{lemma}\label{lSubcat}
Let $X$ be a smooth projective variety over a field $k$ such that $\Db(X)=\langle \ka,E_1,\ldots,E_r \rangle$, where $\ka$ is a phantom category and $E_i$ are exceptional. Then $K_0(\ka\boxtimes E_i)=0$ for all $1\leq i\leq r$.
\end{lemma}

\begin{proof}
We will begin by checking that $\ka$ is a full subcategory of $\ka\boxtimes E$ for any exceptional object $E$.
Define $\Phi\colon \ka\to \ka\boxtimes E$ by $A\mapsto p_1^*(A)\otimes p_2^*E$ on objects and by sending a morphism $f$ to $p_1^*(f)\otimes \id$. Then:
\begin{align*}
\Hom(\Phi(A),\Phi(B))&\cong \Hom_{i+j=0} \Hom(A,B[i])\otimes \Hom(E,E[j])\\
&\cong \Hom(A,B).
\end{align*}
by the K\"unneth formula. 

By Proposition \ref{pGrothExc}, $\ko_X$ is an exceptional object. In particular, by the above argument, $\ka\boxtimes\ko_X\cong \ka$. Now, $\ka\boxtimes \ko_X$ and $\ka\boxtimes E_i$ are abstractly isomorphic as triangulated categories and, therefore, have the same Grothendieck group.

A more conceptual proof goes as follows. We know that 
\[\mathrm{KM}(\ka\boxtimes E_i)\cong \mathrm{KM}(\ka)\otimes\mathrm{KM}(E_i)\cong\mathrm{KM}(\ka)\otimes \mathbbm{1}\cong \mathrm{KM}(\ka) .\]
In particular, the Grothendieck groups of $\ka\boxtimes E_i$ and $\ka$ coincide, which is precisely the claim.
\end{proof}

\begin{proposition}\label{pPhan-uPhan}
Let $X$ be a smooth projective variety over a field $k$ such that $\Db(X)=\langle \ka,E_1,\ldots,E_r \rangle$, where $\ka$ is a phantom category and $E_i$ are exceptional. The following statements are equivalent.
\begin{enumerate}
\item $K_0(X\times X)\cong K_0(X)\otimes K_0(X)$.
\item $K_0(\ka\boxtimes \ka)=0$.
\item $\ka$ is a universal phantom category.
\end{enumerate}
\end{proposition}

\begin{proof}
Recall that there is a semi-orthogonal decomposition
\[\Db(X\times X)=\langle \ka\boxtimes\Db(X),E_1\boxtimes \Db(X),\ldots,E_r\boxtimes \Db(X)\rangle.\]

Note that 
\[\mathrm{KM}(E_i\boxtimes \Db(X))=\mathrm{KM}(E_i)\otimes \mathrm{KM}(\Db(X))=\mathrm{KM}(\Db(X)).\]
In particular, 
\[K_0(E_i\boxtimes \Db(X))=K_0(\Db(X)).\]
 
If $K_0(X)\cong \ZZ^r$ and $K_0(X\times X)\cong K_0(X)\otimes K_0(X)$, then $K_0(X\times X)\cong \ZZ^{r^2}$. Hence, $K_0(\ka\boxtimes \Db(X))=0$. In particular, $K_0(\ka\boxtimes\ka)=0$. 

Conversely, if $K_0(\ka\boxtimes \ka)=0$, then $K_0(\ka\boxtimes\Db(X))=0$ by Lemma \ref{lSubcat}. Since $K_0(E_i\boxtimes\Db(X))\cong \ZZ^r$, we get $K_0(X\times X)\cong K_0(X)\otimes K_0(X)$. Hence, (1) and (2) are equivalent.

For the implication ``(2)$\Rightarrow$(3)'' we use \cite[Prop.\ 4.4]{Gor-Orl}, which, in particular, states that $\ka\subset \Db(X)$ is universal if $K_0(\ka\boxtimes\Db(X))=0$. Since $K_0(\ka\boxtimes E_i)=0$, it is now clear that (2) implies (3). Conversely, if (3) holds, then $\mathrm{KM}(\ka\boxtimes\ka)=\mathrm{KM}(\ka)\otimes \mathrm{KM}(\ka)=0$, hence $K_0(\ka\boxtimes \ka)=0$.
\end{proof}

It seems plausible that the assumption on the Grothendieck groups in the proposition is, at least over $\CC$, automatic, since by Remark \ref{rRatGrothGr} the Grothendieck group of $\ka\boxtimes\ka$ is at most torsion. Surprisingly, it is rather difficult to check that the Grothendieck group of the latter category is indeed trivial. One idea one might have is to use results on Chow groups, but the transfer seems to be subtle. 

\begin{remark}
Let $X$ be as in the proposition. If $\ka$ is a universal phantom category, then $K_0(X\times Y)=K_0(X)\otimes K_0(Y)$ for all smooth projective $Y$.
\end{remark}

The next result deals with phantom categories on surfaces.

\begin{corollary}\label{cPhanUnivphan}
Let $S$ be a complex surface with $p_g=q=0$ which satisfies the Bloch conjecture and admits a phantom category $\ka$. Then $\ka$ is a universal phantom category. 
\end{corollary}

\begin{proof}
By \cite[Rem.\ 3.2]{Gor-Orl}, $K_0(S\times S)\cong K_0(S)\otimes K_0(S)$. It remains to apply Proposition \ref{pPhan-uPhan}.
\end{proof}

\begin{corollary}
Let $S$ be a generic complex determinantal Barlow surface which by \cite{BBKS} admits a semi-orthogonal decomposition of the form $\Db(S)=\langle \ka,\kl_1,\ldots,\kl_{11}\rangle$, where the $\kl_i$ are line bundles and $\ka$ is a phantom category. Then $\ka$ is a universal phantom category. Similarly, if $X_9(2,3)$ is the Dolgachev surface considered in \cite{Cho-Lee}, then the phantom it admits is universal.
Furthermore, the derived categories of the Hilbert schemes $S^{[n]}$ and $(X_9(2,3))^{[n]}$ of $n$ points on these surfaces admit a universal phantom. 
\end{corollary}

\begin{proof}
The first claim is immediate from the previous corollary.

As for the second claim: The Fourier-Mukai transform
\[\mathrm{FM}_\ki\colon \Db(S)\to \Db(S^{[n]})\]
whose kernel is the ideal sheaf $\ki$ of the universal family, is fully faithful by \cite[Thm.\ 1.2]{KS}. The same reasoning holds for \[\mathrm{FM}_\ki\colon \Db(X_9(2,3))\to \Db((X_9(2,3))^{[n]}).\]
It remains to apply Proposition \ref{pEasyresults}. 
\end{proof}

\begin{remark}\label{rPhantomHighDim}
Recall that the structure sheaf of an Enriques surface $Z$ is an exceptional object. Therefore, considering $\PP^2\times S$, $Z\times X_9(2,3)$, $X_9(2,3)\times X_9(2,3)$, $X_9(2,3)\times S$ and $S\times S$, using that $\kappa(X\times Y)=\kappa(X)+\kappa(Y)$ and Lemma \ref{lSubcat}, one can produce fourfolds of Kodaira dimension $-\infty,1,2,3,4$ having a phantom category.

Note that taking the product of $\PP^1$ with any surface containing a phantom gives a threefold of Kodaira dimension $-\infty$ containing a phantom.
\end{remark}

The following is a partial strengthening of \cite[Thm.\ 2.7]{Vial}. The assumption on the field is needed to ensure that $\mathrm{Spec}(K)$ is smooth projective over $k$ for any field extension $K/k$.

\begin{proposition}\label{pPhantom-Lefschetz}
Let $S$ be a smooth projective surface over a perfect field $k$. Assume that $\Db(S)=\langle \ka,E_1,\ldots,E_r\rangle$, where $\ka$ is a universal phantom category and all the $E_i$ are exceptional. Then the integral Chow motive of $S$ is of Lefschetz type, that is, of the form $\mathbbm{1}\oplus \mathbb{L}^{\oplus r}\oplus \mathbb{L}^{\otimes 2}$.
\end{proposition}

\begin{proof}
The statement in \cite{Vial} concerns full exceptional collections. A close inspection of the proof shows that the fullness is only used to ensure that the base change $K_0(S)\to K_0(S_K)$ is surjective for any field extension $K/k$. But since by our assumption $\ka$ is a universal phantom category, $K_0(\ka\boxtimes\Db(\mathrm{Spec}K))=0$, hence $K_0(S)\cong K_0(S_K)$.
\end{proof}

\begin{corollary}
The Chow motive of a complex generic determinantal Barlow surface or of the Dolgachev surface $X_9(2,3)$ is of Lefschetz type over $\ZZ$.\qqed
\end{corollary}

\begin{remark}
The Chow motive of \emph{any} Barlow surface $S$ can be shown to be of Lefschetz type over $\ZZ$. For instance, since the Chow group of zero-cycles on $S$ is universally trivial by \cite[Cor.\ 2.2]{Voi-univ}, the statement follows by \cite[Thm.\ 4.1]{Tot-motive}. 

A different proof can be given as follows. If $S$ is a smooth projective complex surface with $p_g=q=0$ which satisfies the Bloch conjecture, then, by \cite[Prop.\ 2.2]{Gor-Orl}, its Chow motive is of the form
\[M(S)=\mathbbm{1}\oplus \mathbb{L}^{\oplus \rho(S)}\oplus \mathbb{L}^{\otimes 2}\oplus T,\]
where the first Chow group of $T$ is isomorphic to the torsion in $\mathrm{Pic}(S)$ and all its other Chow groups vanish. In particular, if we invert any number $N$ such that the order of $N\cdot\mathrm{Pic}(S)_{\text{tors}}=0$, then over $\mathbb{Z}_{1/N}$ the Chow motive of $S$ is of Lefschetz type; see \cite[Prop.\ 2.3]{Gor-Orl}. It remains to note that the Bloch conjecture holds for a Barlow surface by \cite{Voi-bloch} and set $N=1$.
\end{remark}

\begin{remark}
Let $S$ be a generic determinantal Barlow surface and $S'$ be any surface of general type with $p_g=q=0$ having a quasi-phantom category $\ka'$. Then $\ka\boxtimes\ka'\subset \Db(S\times S')$ is a universal phantom category. This follows immediately from \cite[Thm.\ 1.14]{Gor-Orl}. Note that if we consider $S\times S$, then $\Db(S\times S)$ contains 
\[\langle\ka\boxtimes \ka,\ka\boxtimes E_1,\ldots,\ka\boxtimes E_{11},E_1\boxtimes \ka,\ldots,E_{11}\boxtimes\ka\rangle\]
as an admissible subcategory and this category is a phantom. But the components are as well, hence a phantom category can have admissible subcategories. This fact should be related to the \emph{Noetherianity conjecture} for admissible subcategories, namely the statement that any descending chain of admissible subcategories $\ka_1\supset\ka_2\supset\ldots$ of $\Db(X)$ terminates.
\end{remark}

The following result should be compared to Proposition \ref{pPhantom-Lefschetz}. 

\begin{proposition}
Let $X$ be a smooth projective complex variety of dimension $n$ such that there exists a semi-orthogonal decomposition 
\[\Db(X)=\langle \ka,E_1,\ldots,E_r\rangle,\]
where $E_i$ are exceptional objects and $\ka$ is a universal phantom category. Then the Chow motive of $X$ is of Lefschetz type over $\ZZ[1/(2n!)]$.
\end{proposition}

\begin{proof}
By \cite[Prop.\ 4.4]{Gor-Orl}, the $K$-motive of $\ka$ is trivial. Since $K$-motives are additive with respect to semi-orthogonal decompositions, it follows that the $K$-motive of $X$ is of trivial type. This then remains true over $\ZZ[1/2n!]$. By \cite[Thm.\ 1.4]{Bern-Tab}, this implies that the Chow motive of $X$ is of Lefschetz type over $\ZZ[1/2n!]$. 
\end{proof}

\begin{remark}
It is an interesting question whether in the above situation the Chow motive of $X$ is of Lefschetz type over $\ZZ$. Note that in general even if the $K$-motive of a variety $X$ is of trivial type over $\ZZ$, as is the case in the proposition, the Chow motive of $X$ need not be of Lefschetz type over $\ZZ$, see \cite[Prop.\ 1.7]{Bern-Tab}. However, the example given there does not work over an algebraically closed field.
\end{remark}

\section{Possible further directions}

Recall that if $\kc$ is an admissible subcategory of $\Db(X)$, the composition $\Phi=i\circ i^R$ of the inclusion and its right adjoint, is a Fourier-Mukai functor by \cite[Thm.\ 7.1]{Kuzbase}. Call the kernel $K$ and $i^R$ the projection to $\kc$. Hochschild cohomology of $\kc$ is defined by $\mathsf{HH}^*(\kc)=\Ext^*(K,K)$. If $\ka\subset \Db(X)$ and $\ka'\subset \Db(X')$ are admissible, then the kernel of the projection functor to $\ka\boxtimes \ka'$ is the convolution of the kernels of the projection functors to $\ka$ and $\ka'$, hence, by the K\"unneth formula, 
\[\mathsf{HH}^*(\ka\boxtimes\ka')=\mathsf{HH}^*(\ka)\otimes \mathsf{HH}^*(\ka').\]

In particular, if the restriction maps 
\[\mathsf{HH}^k(X)\to \mathsf{HH}^k(\ka) \quad\text{ and } \quad \mathsf{HH}^l(X')\to\mathsf{HH}^l(\ka')\]
 are isomorphisms for $k\leq \alpha$ and $l\leq \beta$, respectively, then 
 \[\mathsf{HH}^k(X\times X')\cong \oplus_{i+j=k}\mathsf{HH}^i(\ka)\otimes \mathsf{HH}^j(\ka')\]
and for $k\leq \min(\alpha,\beta)$ the Hochschild cohomology of $\ka$ and $\ka'$ determines that of $X\times X'$.

If $\ka$ and $\ka'$ are (quasi-)phantoms appearing as complements to exceptional collections on surfaces, we can sometimes get better results. For this we need to recall some notions.

Let $(E_1,\ldots,E_n)$ be an exceptional collection on a variety $X$. Its \emph{anticanonical pseudoheight} is defined as
\begin{align*}
\mathrm{ph}_{\mathrm{ac}}(E_1,\ldots,E_n)&=\min\limits_{1\leq a_0<a_1<\ldots <a_p\leq n}\big(e(E_{a_0},E_{a_1})+\ldots \\
&+e(E_{a_{p-1}},E_{a_p})+e(E_{a_p},E_{a_0}\otimes \omega_X^{-1})-p\big),
\end{align*}
where $e(F,G)=\min\{p\in \ZZ\mid \Hom^p(F,G)\neq 0\}$. Let us call the numbers over which the minimum is taken the \emph{length-sum} of a $p$-chain. The \emph{pseudoheight} of the collection is 
\[\mathrm{ph}(E_1,\ldots,E_n)=\mathrm{ph}_{\mathrm{ac}}(E_1,\ldots,E_n)+\dim(X).\]

The \emph{(anti-)canonically extended collection} is defined by 
\[(E_1,\ldots,E_n,E_{n+1},\ldots,E_{2n}):=(E_1,\ldots,E_n,E_1\otimes \omega_X^{-1},\ldots,E_n\otimes\omega_X^{-1}).\]

The extended collection is called \emph{Hom-free} if $\Hom^p(E_i,E_j)=0$ for all $p\leq 0$ and for all $1\leq i<j\leq i+n$.

\begin{lemma}
Let $(E_1,\ldots,E_n)$ and $(F_1,\ldots,F_m)$ be exceptional collections on $X$ and $Y$, respectively. If both extended collections are Hom-free, then the collection $(E_1\boxtimes F_1,\ldots, E_n\boxtimes F_m)$ on $X\times Y$ is also Hom-free. The pseudoheight of the collection on the product is at least $1+\dim(X)+\dim(Y)$.
\end{lemma}

\begin{proof}
The first statement follows from the K\"unneth formula
\[\Hom^p(A\boxtimes B,C\boxtimes D)=\bigoplus\limits_{s+t=p}\Hom^s(A,C[s])\times \Hom^t(B,D[t]),\]
where $A,C\in \Db(X)$ and $B,D\in \Db(Y)$, and, of course, we use that $\omega_X\boxtimes\omega_Y=\omega_{X\times Y}$. The second statement follows from \cite[Lem.\ 4.10]{Kuz15}.
\end{proof}

The interest in the above notions stems from \cite[Cor.\ 4.6]{Kuz15} which says that if $\ka$ is the orthogonal complement of a given exceptional collection, then the restriction morphism on Hochschild cohomology $\mathsf{HH}^k(X)\to \mathsf{HH}^k(\ka)$ is an isomorphism for $k\leq \mathrm{ph}(E_1,\ldots,E_n)-2$ and an injection for $k\leq \mathrm{ph}(E_1,\ldots,E_n)-1$. In particular, we can apply this to the examples of quasi-phantom categories appearing in the literature. The following is a(n incomplete) list of examples, where the pseudoheight of the respective exceptional collection was computed (for (1)-(3) in \cite{Kuz15}):

\begin{enumerate}
\item $S$ is the classical Godeaux surface, $K_0(\ka)=\ZZ/5\ZZ$ (see \cite{BBS12}), the length of the exceptional sequence is $11$ and the pseudoheight is $3$;
\item $S$ is a Beauville surface, $K_0(\ka)=(\ZZ/5\ZZ)^3$ (see \cite{GS13}), the maximal length of an exceptional sequence is $4$, and the pseudoheight is $4$ or $3$, depending on the collection;
\item $S$ is a Burniat surface, $K_0(\ka)=(\ZZ/2\ZZ)^6$ (see \cite{A-O12}), the length of the exceptional sequence is $6$ and the pseudoheight is $4$;
\item $S$ is a surface isogeneous to a product, $K_0(\ka)=(\ZZ/3\ZZ)^5$, the length of the exceptional sequence is $4$ and the pseudoheight is $4$ (see \cite{Lee15}).
\end{enumerate}

Interestingly, given exceptional collections on $X$ and $Y$, it is not quite straightforward to compute the pseudoheight of the box product collection out of the pseudoheights of the original collections. First of all, note that
\[e(E\boxtimes F,G\boxtimes H)=e(E,G)+e(G,H).\]
Given a $p$-chain of length-sum $\alpha$ on $X$ and a $q$-chain of length-sum $\beta$ on $Y$, one easily sees that length-sum of the corresponding chain on $X\times Y$ is
\[(p+1)(\beta+q)+(q+1)(\alpha+p)-pq-p-q=p\beta+q\alpha+\beta+\alpha+pq.\] 
So, it is not obvious what the pseudoheight of the collection on the product is.

But at least in cases (2)-(4) the collections are Hom-free, hence we get

\begin{proposition}
Let $S, S'$ be two surfaces from items (2)-(4) on the above list. Consider the semi-orthogonal decomposition
\[\Db(S\times S')=\langle \ka_S\boxtimes\ka_{S'},\ka_S\boxtimes F_1,\ldots, E_1\boxtimes \ka_{S'},\ldots, E_1\boxtimes F_1,\ldots,E_k\boxtimes F_l\rangle.\]
Then 
\[\mathsf{HH}^k(\ka_S\boxtimes\ka_{S'},\ka_S\boxtimes F_1,\ldots,\ka_S\boxtimes F_l, E_1\boxtimes \ka_{S'},\ldots,E_k\boxtimes \ka_{S'})\cong\mathsf{HH}^k(S\times S')\]
for $k\leq 3$.\qqed
\end{proposition}

Note that the subcategories $\ka_S\boxtimes F_i$ and $E_j\boxtimes \ka_{S'}$ are completely orthogonal for all $i,j$. Also note that $K_0(\ka_S\boxtimes F_i)=K_0(\ka_S)$ and $K_0(E_j\boxtimes \ka_{S'})=K_0(\ka_{S'})$, compare Lemma \ref{lSubcat}. It would be interesting to see whether there is ``more'' information, for example about deformations of $S\times S'$ which are described by $\mathsf{HH}^2(S\times S')$, contained in these quasi-phantoms or in the phantom category $\ka_S\boxtimes\ka_{S'}$.
\smallskip

Let us conclude with a list of questions one could also consider in this context. In the following we tacitly assume that the base field is $\CC$.

\begin{enumerate}
\item \textbf{Question 1:} If the Hochschild homology of an admissible subcategory of $\Db(X)$ is zero, is its Grothendieck group automatically finite or at least torsion?
\item
\textbf{Question 2:} Can a (quasi-)phantom category have a bounded t-structure?
\item
\textbf{Question 3:} In dimension $2$ (quasi-)phantom categories presumably do not appear if the Kodaira dimension is $0$. Does this hold in higher dimensions?
\item
\textbf{Question 4:} Can one describe the group of autoequivalences of a (quasi-)phantom category?
\item
\textbf{Question 5:} Is there a phantom category which is not a universal phantom category?
\item
\textbf{Question 6:} Does there exist a full exceptional collection on a Barlow or a Dolgachev surface?
\end{enumerate}

Note that it was proved in \cite[Thm.\ 5.5]{Gor-Orl} that the vanishing of $K_0$ over $\QQ$ automatically implies the vanishing of Hochschild homology (and the higher $K$-groups) of an admissible subcategory $\ka$. Hence, (1) asks for a converse.
\smallskip

Recall that a stability condition, see \cite{Bri-stab}, consists of a heart $\kc$ of a bounded t-structure on $\Db(X)$ and a group homomorphism $Z\colon K_0(\kc)=K_0(X)\to \CC$ which is assumed to satisfy some axioms. Among them is the following: $Z(T)\neq 0$ for any $0\neq T\in \kc$. Hence, there cannot be a stability condition on a quasi-phantom category. Note that the same argument shows that there cannot be a stability condition on a variety whose Grothendieck group has torsion. 
\smallskip

With regards to Question 3: most examples so far were constructed in Kodaira dimension $2$. In the recent preprint \cite{Cho-Lee} the authors construct a phantom on a Dolgachev surface which has Kodaira dimension $1$. Since K3 and abelian surfaces are Calabi-Yau, they do not admit any semi-orthogonal decompositions (this statement uses \cite[Prop.\ 3.10]{Huy-book}, which is due to Bridgeland) and the same holds for bielliptic surfaces by \cite[Thm.\ 1.7]{Kaw-Okawa}. It is not quite clear that an Enriques surface does not admit a (quasi-)phantom category, but at least it does not admit an almost full exceptional collection by \cite[Thm.\ 3.13]{Vial}.
\smallskip

Most examples of (quasi-)phantoms so far were constructed on surfaces $S$ with ample canonical bundle, so the group $\Aut(\Db(S))$ of autoequivalences of $S$ is a (semi-direct) product of automorphisms, line bundle twists and powers of the shift functor by \cite{Bon-Orl}. It would be interesting to understand whether the group of autoequivalences of a (quasi-)phantom category reflects some geometric aspects of the underlying variety.
\smallskip

Finally, the last question is connected to the folklore conjecture that the existence of a full exceptional collection implies rationality of a surface.

\end{document}